\documentclass[12pt,leqno,draft]{article} 
\usepackage{amsmath,amssymb,amsthm,amsfonts,bm}
\usepackage{enumerate,color}
\makeatletter
\def\@cite#1#2{[{{\bfseries #1}\if@tempswa , #2\fi}]}
\renewcommand{\section}{%
\@startsection{section}{1}{\z@}
{0.5truecm plus -1ex minus -.2ex}%
{1.0ex plus .2ex}{\bfseries\large}}
\def\@seccntformat#1{\csname the#1\endcsname.\ }
\makeatother

\numberwithin{equation}{section} 
\newtheorem{thm}{Theorem}[section]
\newtheorem{corollary}[thm]{Corollary}
\newtheorem{lem}[thm]{Lemma}

\theoremstyle{definition}
\newtheorem{df}{Definition}[section]
\newtheorem{remark}{Remark}[section]
\newtheorem*{prth1.1}{Proof of Theorem 1.1}

\newcommand{\ep}{\varepsilon}
\newcommand{\pa}{\partial}

\newcommand{\norm}[4]{\Vert #1 \Vert _{{#2}^{#3}(#4)}}

\newcommand{\h}{H}
\newcommand{\tmax}{T_{{\rm max},\ep}}
\newcommand{\lp}[2]{\|#2\|_{L^{#1}(\Omega)}}

\newcommand{\nep}{n_{\ep}}
\newcommand{\cep}{c_{\ep}}
\newcommand{\uep}{u_{\ep}}
\newcommand{\io}{\int_\Omega}
\newcommand{\nha}{\widetilde{\nep}}
\newcommand{\cd}{(\cdot,t)}
\newcommand{\iio}{\int_0^T \io}

\newcommand{\abs}{\medskip}

\def\D{\Delta}
\def\om{\Omega}
\def\na{\nabla}
\def\vp{\varphi}
\def\e{\varepsilon}
\def\h{\hspace}
\def\etmax{T_{{\rm max},\e}}

\def\norm#1{\|#1\|}
 
\begin{document}

\footnote[0]
    {2010{\it Mathematics Subject Classification}\/. 
    Primary: 35K55; Secondary: 92C17; 35Q35.
    }
\begin{center}
    \Large{{\bf 
    Global weak solutions to a 3-dimensional  
    degenerate \\and singular chemotaxis-Navier--Stokes system \\
    with logistic source
               }}
\end{center}
\vspace{5pt}
\begin{center}
    Shunsuke Kurima\\
    \vspace{2pt}
    Department of Mathematics, 
    Tokyo University of Science\\
    1-3, Kagurazaka, Shinjuku-ku, Tokyo 162-8601, Japan\\
    {\tt shunsuke.kurima@gmail.com}\\
    \vspace{12pt}
     Masaaki Mizukami\footnote{Partially supported by 
    JSPS Research Fellowships 
    for Young Scientists (No.\ 17J00101).}\footnote{Corresponding author}\\
    \vspace{2pt}
    Department of Mathematics, 
    Tokyo University of Science\\
    1-3, Kagurazaka, Shinjuku-ku, Tokyo 162-8601, Japan\\
    {\tt masaaki.mizukami.math@gmail.com}\\
\end{center}
\begin{center}    
    \small \today
\end{center}

\vspace{2pt}
\newenvironment{summary}
{\vspace{.5\baselineskip}\begin{list}{}{%
     \setlength{\baselineskip}{0.85\baselineskip}
     \setlength{\topsep}{0pt}
     \setlength{\leftmargin}{12mm}
     \setlength{\rightmargin}{12mm}
     \setlength{\listparindent}{0mm}
     \setlength{\itemindent}{\listparindent}
     \setlength{\parsep}{0pt} 
     \item\relax}}{\end{list}\vspace{.5\baselineskip}}
\begin{summary}
{\footnotesize {\bf Abstract.}
This paper considers the degenerate and singular chemotaxis-Navier--Stokes system 
with logistic term 
\begin{equation*}
     \begin{cases}
         n_t + u\cdot\nabla n =
           \Delta n^m - \chi\nabla\cdot(n\nabla c) 
          + \kappa n -\mu n^2, 
         &x\in \Omega,\ t>0,
 \\[2mm]
         c_t + u\cdot\nabla c = 
         \Delta c - nc,
          &x \in \Omega,\ t>0,
 \\[2mm]
        u_t + (u\cdot\nabla)u 
          = \Delta u + \nabla P + n\nabla\Phi, 
            \quad \nabla\cdot u = 0, 
          &x \in \Omega,\ t>0, 
     \end{cases} 
 \end{equation*} 
where $\Omega\subset \mathbb{R}^3$ is a bounded domain 
and $\chi,\kappa \ge 0$ and $m, \mu >0$. 
In the above system without fluid environment Jin \cite{Jin-2017} 
showed existence and boundedness of global weak solutions. 
On the other hand, in the above system with $m=1$,  
Lankeit \cite{Lankeit_2016} established global existence of weak solutions. 
However, 
the above system with $m>0$   
has not been studied yet. The purpose of this talk is to establish global existence 
of weak solutions in the chemotaxis-Navier--Stokes system 
with degenerate diffusion and logistic term.
}
\end{summary}
\vspace{10pt}

\newpage

\section{Introduction and results}
In the study of partial differential equations 
mathematical models describing the natural phenomena, 
e.g., the heat equation, the Fisher--KPP equation and so on, 
are one of the important topics in the mathematical analysis, 
and studied by many mathematicians. 
Recently, there have been many variations of systems of partial differential equations which describe complicated phenomena. 
One of systems which describe important biological phenomena related to animals life 
is the Keller--Segel system  
\begin{align*}
n_t&=\Delta n^m -\chi\nabla\cdot(n^{q-1} \nabla c),\\ 
c_t&=\Delta c -c+n,
\qquad x\in \Omega,\ t>0, 
\end{align*}
where $\Omega \subset \mathbb{R}^N$ ($N\in \mathbb{N}$), $m>0$, 
$\chi\ge 0$, $q \ge 2$, 
which describes migration of species by chemotaxis which is the property such that 
species move towards higher concentration of the chemical substance. 
Here 
$\Delta n^m$ with $m=1$, that is, $\Delta n$ 
is called a {\it linear diffusion} and 
$\Delta n^m$ with $m\neq 1$ is called a {\it nonlinear diffusion}.  
In other words, the case that $m>1$ is said to be a {\it degenerate diffusion} or 
a {\it porous medium diffusion}, and the case that $0 < m < 1$ is said to be a 
{\it singular diffusion} or a {\it fast diffusion}.  
This system with $m=1$ and $q=2$ was first proposed by Keller--Segel \cite{K-S} 
and then the system with $m,q\in \mathbb{R}$ was suggested by Hillen--Painter \cite{Hillen_Painter_2009}. 
In the case that $m=1$ and $q=2$ 
it is known that the size of initial data determines behaviour of solutions 
to the 2-dimensional Keller--Segel system. 
More precisely, there exists some constant $C > 0$ such that 
if an initial data $n_0$ satisfies $\lp{1}{n_0} < C$ 
then global bounded classical solutions exist (\cite{Nagai-Senba-Yoshida}), 
moreover, for any $m > C$ there exist initial data such that $\lp{1}{n_0} = m$ 
and the corresponding solution blows up in finite time 
(\cite{Horstmann-Wang,Mizoguchi-Winkler}),  
where $C>0$ can be given as $C = \frac{8\pi}{\chi}$ in the radial setting and 
$C = \frac{4\pi}{\chi}$ in the nonradial setting. 

\abs
In the other dimensional case we can see that there are many global/blow-up solutions. 
In the case that $m=1,q=2$ and $N=1$ Osaki--Yagi \cite{Osaki-Yagi} showed 
that global existence and boundedness hold for all smooth initial data:  
This means that there is not a blow-up solution in the 1-dimensional setting. 
On the other hand, it is known that the 3-dimensional case has many blow-up solutions: 
Winkler \cite{Winkler_2013_blowup} established that for all $m>0$ there are initial data $n_0$ such that $\lp{1}{n_0}=m$ and the corresponding solution blows up in finite time. 
To obtain global existence of classical solutions we need some additional conditions for 
initial data: Winkler \cite{win_aggregationvs} and Cao \cite{Xinru_higher} 
proved existence of global boundedness classical solutions 
under conditions that 
initial data $(n_0,c_0)$ is 
small enough with respect to a suitable Lebesgue norm. 

\abs
On the other hand, in the case that $m\ge 1$ and $q\ge 2$, 
it is known that relations between $m$ and $q$ determine 
whether solutions of the Keller--Segel system 
exist globally or not; 
in the case that $m > q - \frac{2}{N}$ Ishida--Seki--Yokota \cite{Ishida-Seki-Yokota} 
obtained global existence of solutions; 
on the other hand, in the case that $m < q- \frac 2N$ 
Ishida--Yokota \cite{Ishida-Yokota_2013} proved that 
there exist initial data such that the corresponding solution blows up 
in finite or infinite time, and recently Hashira--Ishida--Yokota \cite{Hashira-Ishida-Yokota} found 
initial data such that the solution blows up in finite time.   

\abs
As a generalized problem, the nonlinear Keller--Segel system with logistic source 
which the first equation in the above system is replaced with
\begin{align*}
n_t=\Delta n^m -\chi\nabla\cdot(n^{q-1}\nabla c) +\kappa n -\mu n^2,  
\end{align*}
where $m,\mu>0$, $\chi,\kappa \ge 0$, $q\ge 2$,  
was also studied, and it is shown that 
the logistic source $\kappa n -\mu n^2$ suppresses blow-up phenomenon. 
When $m=1$ and $q=2$, in the 2-dimensional setting 
Osaki et al.\ \cite{OTYM} obtained  
global existence and boundedness 
for all smooth initial data, 
and in the higher dimensional setting 
Winkler \cite{Winkler_2010_logistic} proved  
existence of global classical solutions under largeness conditions for $\mu>0$; 
in the Keller--Segel system with logistic source global existence of solutions holds   
even though the $L^1$-norm of the initial data is large enough.   
Moreover, Lankeit \cite{lankeit_evsmoothness} established global existence of weak solutions 
without largeness condition for $\mu >0$. 
On the other hand, in the case that $q=2$ a resent result established by 
Zheng--Wang \cite{Zheng-Wang_2016} 
asserted that the condition that 
$m > 1+ \frac{(N-2)_+}{N+2}$ derives global existence and boundedness.  

\abs
As we confirmed before, 
in the Keller--Segel system 
the relation between $m$ and $q$ 
strongly affects behaviour of solutions, 
and 
the logistic source often relaxes conditions for global existence. 
Thus ``what is the condition for $m$ to derive existence of global/blow-up solutions?'' 
makes one of the main topics. 
This will be also one of the important topics in the study of 
the chemotaxis-(Navier--)Stokes system 
with logistic source 
\begin{align*}
  n_t + u\cdot\nabla n 
&
  =\Delta n^m -\chi\nabla\cdot(n^{q-1}\nabla c) 
  + \kappa n - \mu n^2, 
\\
  c_t + u\cdot\nabla c 
& 
  =\Delta c -nc, 
\\ 
  u_t + \lambda(u\cdot\nabla)u 
& 
  = \Delta u +\nabla P +n\nabla \Phi, 
\ 
  \nabla \cdot u=0, 
\quad 
  x\in \Omega,\ t>0, 
\end{align*}
where $m,\mu >0$, $\chi ,\kappa \ge0$, $q\ge 2$, $\lambda =0$ (the chemotaxis-Stokes system)  
or $\lambda = 1$ (the chemotaxis-Navier--Stokes system). 
In the case that $m=1$, $q=2$ and moreover $\kappa=\mu=0$ 
Winkler established global existence of classical solutions 
in the 2-dimensional setting (\cite{W-2012}) 
and obtained existence of global weak solutions 
in the 3-dimensional setting (\cite{Winkler_2016}).   
In the case that $q=2$ and $\kappa=\mu=0$ 
Tao--Winkler \cite{Tao-Winkler_2012_KSF} 
and Chung--Kang \cite{Chung-Kang_2016} 
showed global existence under the condition that $m>1$.  

\abs
On the other hand, in the chemotaxis-(Navier--)Stokes system with logistic source 
results similar to those in the Keller--Segel system with logistic source hold. 
In the case of the fluid-free system 
Lankeit--Wang \cite{Lankeit-Wang_2017} showed global existence and boundedness 
in the system with $m=1$ 
under some largeness conditions for $\mu>0$, 
and Jin \cite{Jin-2017} proved global existence of weak solutions 
of the system with $m>1$. 
In the system with fluid equation and $m=1$, $q=2$ 
global existence and boundedness of classical solutions hold 
in the 2-dimensional setting (cf.\ \cite{HKMY_1}) 
and existence of global bounded classical solutions to the 
chemotaxis-Stokes system holds under some largeness conditions for $\mu>0$
in the 3-dimensional setting (cf. \cite{CKM}),   
and moreover Lankeit \cite{Lankeit_2016} showed that 
global existence of weal solutions in the system without largeness conditions for $\mu>0$. 
However, the chemotaxis-Navier--Stokes system with degenerate diffusion and logistic source 
has never been tried; which means that conditions of $m$ for global existence of weak solutions 
are still open problem. 
The purpose of this paper is to 
derive some condition of $m$  
for existence of global weak solutions 
in the degenerate chemotaxis-Navier--Stokes system with logistic source. 
Here we note that the methods in \cite{Jin-2017} cannot be applied to 
this problem because of the difficulty of the Navier--Stokes equation (explained later). 

%
%
\abs
In order to attain this purpose 
we consider the following degenerate and singular chemotaxis-Navier--Stokes system 
with logistic term:
 \begin{equation}\label{P}
     \begin{cases}
         n_t + u\cdot\nabla n = 
         \Delta n^m - \chi\nabla\cdot(n\nabla c) 
          + \kappa n -\mu n^2,
         &x\in \Omega,\ t>0,
 \\[2mm]
         c_t + u\cdot\nabla c = \Delta c -nc,
          &x \in \Omega,\ t>0,
 \\[2mm]
        u_t  + (u\cdot\nabla) u 
          = \Delta u + \nabla P + n\nabla\Phi, 
            \quad \nabla\cdot u = 0, 
          &x \in \Omega,\ t>0, 
 \\[2mm]
        \partial_\nu n^m = \partial_\nu c = 0, \quad 
        u = 0, 
        &x \in \partial\Omega,\ t>0, 
 \\[2mm]
        n(x,0)=n_{0}(x),\ 
        c(x,0)=c_0(x),\ u(x,0)=u_0(x), 
        &x \in \Omega,
     \end{cases}
 \end{equation}
\noindent
where $\Omega$ is a bounded domain 
in $\mathbb{R}^3$ with smooth boundary $\partial\Omega$ and 
$\partial_\nu$ denotes differentiation with respect to the 
outward normal of $\partial\Omega$; 
$\chi, \kappa\ge 0$ and 
$\mu, m>0$ are constants;  
$n_0, c_0, u_0, \Phi$ 
are known functions satisfying
 \begin{align}\label{condi;ini1}
   &0 < n_0 
   \in C(\overline{\Omega}), 
 \quad 
   0 < c_0 \in W^{1,q}(\Omega), 
 \quad 
   u_0 \in D(A^{\theta}), \\ \label{condi;ini2}
   &\Phi \in C^{1+\beta}(\overline{\Omega}) 
 \end{align}
for some $q > 3$, 
$\theta \in \left(\frac{3}{4}, 1\right)$, 
$\beta > 0$ and 
$A$ denotes the realization of the Stokes operator 
under homogeneous Dirichlet boundary conditions 
in the solenoidal subspace $L_{\sigma}^2(\Omega)$ of $L^2(\Omega)$.

Before stating the main theorem, we define weak solutions of \eqref{P}. 
\begin{df}\label{def;weaksol}
 A  triplet $(n, c, u)$ is called 
  a $($global$)$ {\it weak solution} of \eqref{P} if 
  \begin{align*}
    n 
    &\in L^2_{\rm loc}([0,\infty);L^2(\Omega)), 
   \\
   n^m
    & \in L^{\frac{4}{3}}_{\rm loc}([0,\infty);W^{1,\frac{4}{3}}(\Omega)),
   \\
    c
    &\in L^2_{\rm loc}([0,\infty);W^{1,2}(\Omega)),
   \\
    u
    &\in L^2_{\rm loc}([0,\infty);W^{1,2}_{0,\sigma}(\Omega))
  \end{align*}
  and the identities
  \begin{align*}
    &-\int^\infty_0\!\!\!\!\int_\Omega n\vp_t
     -\int_\Omega n_0\vp(\cdot,0)
     -\int^\infty_0\!\!\!\!\int_\Omega n u\cdot\na\vp
   \\
    &\h{7.0mm}=-\int^\infty_0\!\!\!\!\int_\Omega\na n^{m} \cdot\na\vp
      +\chi\int^\infty_0\!\!\!\!\int_\Omega n \na c\cdot\na\vp
      +\int^\infty_0\!\!\!\!\int_\Omega (\kappa n - \mu n^2)\vp,
   \\[1.5mm]
    &-\int^\infty_0\!\!\!\!\int_\Omega c\vp_t
     -\int_\Omega c_0\vp(\cdot,0)
     -\int^\infty_0\!\!\!\!\int_\Omega cu\cdot\na\vp
   =-\int^\infty_0\!\!\!\!\int_\Omega\na c\cdot\na\vp
      -\int^\infty_0\!\!\!\!\int_\Omega nc\vp,
   \\[1.5mm]
    &-\int^\infty_0\!\!\!\!\int_\Omega u\cdot\psi_t
     -\int_\Omega u_0\cdot\psi(\cdot,0)
     -\int^\infty_0\!\!\!\!\int_\Omega u\otimes u\cdot\na\psi
   \\
    &\h{7.0mm}=-\int^\infty_0\!\!\!\!\int_\Omega\na u\cdot\na\psi
     +\int^\infty_0\!\!\!\!\int_\Omega n\na\psi\cdot\na\Phi
  \end{align*}
  hold for all $\vp\in C^{\infty}_0(\overline{\Omega}\times [0,\infty))$
  and all $\psi\in C^{\infty}_{0,\sigma}(\Omega \times [0,\infty))$,
  respectively.
\end{df}

The main result reads as follows. 
The following theorem gives existence of global weak solutions to \eqref{P}. 
%
\begin{thm}\label{mainthm1}
  Let $\Omega\subset\mathbb{R}^3$ be a bounded smooth domain and   
  let $\chi, \kappa \geq 0$ and 
  $\mu, m>0$. 
  Assume that $n_0, c_0, u_0$ and $\Phi$ satisfy \eqref{condi;ini1}{\rm --}\eqref{condi;ini2} 
  with some $q>3$, $\theta \in(\frac{3}{4},1)$ and $\beta \in (0,1)$.
  Then there exists a weak solution of \eqref{P},
  which can be approximated by a sequence of solutions 
  $(n_{\ep} ,c_{\ep}, u_{\ep})$ of an approximate problem  $($see Section \ref{Sec2}\/$)$ 
  in a pointwise manner.
\end{thm}

\begin{remark}
This result means existence of weak solutions to \eqref{P} for {\it all} $m>0$;  
which implies that we could construct weak solutions of \eqref{P} 
in not only the case that $m>1$ (the case of a porous medium diffusion) 
but also $0 < m <1$ (the case of a fast diffusion). 
\end{remark}

The proof of Theorem \ref{mainthm1} can be applied to 
a nondegenerate chemotaxis-Navier--Stokes system which namely 
is the case that $\Delta n^m$ is replaced with $\Delta (n+1)^m$, 
and enables us to see the following result. 

\begin{corollary}\label{cor}
  Let $\Omega\subset\mathbb{R}^3$ be a bounded smooth domain and   
  let $\chi, \kappa \geq 0$ and 
  $ \mu, m>0$. 
  Assume that $n_0, c_0, u_0$ and $\Phi$ satisfy \eqref{condi;ini1}{\rm --}\eqref{condi;ini2} 
  with some $q>3$, $\theta \in(\frac{3}{4},1)$ and $\beta \in (0,1)$.
  Then there exists a weak solution of 
  the nondegenerate chemotaxis-Navier--Stokes system.
\end{corollary}

The strategy for the proof of Theorem \ref{mainthm1} 
is described as follows.
We start with the construction of local approximate solutions of \eqref{P}. 
We next derive estimates for the approximate solution. 
Thanks to the estimates, we extend the local approximate solution globally in time. 
Finally, passing to the limit from the global approximate solution, 
we obtain the desired global weak solution. %
In the previous work by Jin \cite{Jin-2017} which deals with the fluid-free case,  
aided by uniform $L^\infty $-estimates, 
we can attain convergences of approximate solutions.  
Now, because of the difficulty of the Navier--Stokes equation, 
we could not expect $L^\infty$-boundedness of solutions; 
thus we have to need different methods to consider this problem.  
More precisely, we rely on the Lions--Aubin lemma. 
In this strategy the key is to establish estimates for $\nabla n_{\ep}$, 
where $n_{\ep}$ is the solution of the first equation in the approximate problem 
and $\ep$ is a positive parameter. 
More precisely, generalizing calculations in \cite{Lankeit_2016}, 
we first obtain 
\begin{equation}\label{test}
  \int_{0}^{T}\int_{\Omega}
  \frac{(n_{\ep}+\ep)^{m-1}|\nabla n_{\ep}|^{2}}{n_{\ep}} 
=
  \left( \frac{2}{m+1} \right)^2
  \int_{0}^{T}\int_{\Omega}
  \frac{|\nabla (n_{\ep}+\ep)^{\frac{m+1}{2}}|^{2}}{n_{\ep}}  
\leq C_1(T)
\end{equation}
for all $\ep \in (0, 1)$ with some constant $C_1(T)>0$. 
Aided by this estimate we can have that $((\nep+ \ep)^\frac{m+1}{2} )_{\ep\in (0,1)}$ 
is bounded in $L^\frac{4}{3}([0,T);W^{1,\frac{4}{3}}(\Omega))$; however, it seems to be 
difficult to obtain the estimate for $\pa_t (\nep+\e)^\frac{m+1}{2}$ for all $m>0$.  
Thus we need additional estimates for approximate solutions. 
Here the inequality \eqref{test} ensures that 
\begin{align*}
\int_{0}^{T}\int_{\Omega}(n_{\ep}+\ep)^{m-2}|\nabla n_{\ep}|^{2}
\leq 
\int_{0}^{T}\int_{\Omega}
\frac{(n_{\ep}+\ep)^{m-1}|\nabla n_{\ep}|^{2}}{n_{\ep}} 
\leq C_1(T) 
\end{align*}
for all $\ep \in (0, 1)$. 
This together with the identity  
\[
\int_{0}^{T}\int_{\Omega}|\nabla (n_{\ep}+\ep)^{\frac{m}{2}}|^{2} 
= 
\frac{m^2}4
\int_{0}^{T}\int_{\Omega}(n_{\ep}+\ep)^{m-2}|\nabla n_{\ep}|^{2}
\]
means that 
$\bigl((n_{\ep}+\ep)^{\frac{m}{2}}\bigr)_{\ep \in (0, 1)}$ 
is bounded in $L^2_{\rm loc}([0, \infty); W^{1, 2}(\Omega))$. 
We moreover see that 
$$
\|\pa_t (n_{\ep}+\ep)^{\frac{m}{2}}  \|_{L^1(0, T; (W^{2, 4}_{0}(\Omega))^{*})} 
\leq C_2(T)
$$
for all $\ep \in (0, 1)$ with some $C_2(T)>0$, 
which derives that
$(\partial_{t}(n_{\ep}+\ep)^{\frac{m}{2}})_{\ep \in (0, 1)}$ 
is bounded in $L^1(0, T; (W^{2, 4}_{0}(\Omega))^{*})$.  
Then, aided by the Lions--Aubin theorem, 
we can show convergences of 
solutions of the approximation of \eqref{P} 
and we can prove Theorem \ref{mainthm1}.  

The plan of this paper is as follows. 
In Section \ref{Sec2} we introduce the approximate problem \eqref{Pe} 
of \eqref{P}
and establish global existence in \eqref{Pe}. 
In Section \ref{Sec3} 
we show the several estimates for solutions to the approximate problem of \eqref{P}. 
Finally, in Section \ref{Sec4} 
we prove Theorem \ref{mainthm1} by passage to the limit 
in the approximate problem via estimates from Section \ref{Sec3}. 


\section{Global existence in an approximate problem}\label{Sec2}

We start by considering the following approximate problem with parameter $\e >0$:  
\begin{equation}\label{Pe}
  \begin{cases}
     n_{\ep t}+u_\e\cdot\na n_{\ep} 
     =   \Delta (n_{\ep}+\ep)^m
     -\chi\na\cdot\big(\frac{n_{\ep}}{1+\e n_{\ep}}\na c_\e\big)
     +\kappa n_{\ep} -\mu n_{\ep}^2,
  \\[2mm]
     c_{\e t}+u_\e\cdot\na c_\e
     =\D c_\e
     -c_\e\frac{1}{\e}\log\big(1+\e n_{\ep}\big),
  \\[2mm]
     u_{\e t}+(Y_\e u_\e\cdot\na)u_\e
     =\D u_\e
     +\na P_\e
     + n_{\ep}\na\Phi,
     \quad\na\cdot u_\e=0,
  \\[2mm]
     \partial_\nu n_{\ep}|_{\partial\om}
     =\partial_\nu c_\e|_{\partial\om}=0,
     \quad u_\e|_{\partial\om}=0,
  \\[2mm]
     n_{\ep}(\cdot,0)=n_{0},\quad
     c_\e(\cdot,0)=c_0,\quad
     u_\e(\cdot,0)=u_0,
  \end{cases}
\end{equation}
where $Y_\e=(1+\e A)^{-1}$. 
In this section we shall show global existence of solutions to the approximate problem \eqref{Pe}. 
We first give the following result which states 
local existence in \eqref{Pe}. 

\smallskip

%
\begin{lem}\label{localsol}
  Let $\chi, \kappa \geq0$, 
  $\mu>0$, $m>0 $ and let 
  $\Phi\in C^{1+\beta}(\overline{\om})$ for some $\beta \in (0,1)$.  
  Assume that $n_0,c_0,u_0$ satisfy \eqref{condi;ini1} 
  with some $q>3,\theta \in(\frac{3}{4},1)$.  
  Then for each $\e > 0$ there exist $\etmax$ and 
  uniquely determined functions\/{\rm :}
    \begin{align*}
      n_{\ep}
      &\in C^0(\overline{\om}\times[0,\etmax))
       \cap C^{2,1}(\overline{\om}\times(0,\etmax)),
    \\
       c_\e
      &\in C^0(\overline{\om}\times[0,\etmax))
       \cap C^{2,1}(\overline{\om}\times(0,\etmax))
       \cap L^\infty_{\rm loc}([0,\etmax);W^{1,q}(\om)),
    \\
       u_\e
      &\in C^0(\overline{\om}\times[0,\etmax)) 
       \cap C^{2,1}(\overline{\om}\times(0,\etmax)),
    \end{align*}
  which together with some 
  $P_\e\in C^{1,0}(\overline{\om}\times(0,\etmax))$ 
  solve \eqref{Pe} classically. 
  Moreover, $n_{\ep}$ and $c_\e$ are positive 
  and the following alternative holds\/{\rm :} 
  $\etmax=\infty$ or
  \begin{align*}
     \norm{n_{\e}(\cdot,t)}_{L^{\infty}(\om)}
    +\norm{c_\e(\cdot,t)}_{W^{1,q}(\om)}
    +\norm{A^\theta u_\e(\cdot,t)}_{L^2(\om)}
    \to  \infty
  \end{align*}
as $t\nearrow \etmax$.
\end{lem}
\begin{proof}
This lemma can be shown by a standard fixed point theorem with 
parabolic regularity arguments. More precisely, 
a combination of the proofs of \cite[Lemma 2.1]{Tao-Winkler_2011_non} 
and \cite[Lemma 2.1]{W-2012} enables us to obtain local existence in \eqref{Pe}. 
\end{proof}
%

In the following for all $\e\in (0,1)$ 
we denote by $(\nep,\cep,\uep)$ the corresponding solution of \eqref{Pe} given by 
Lemma \ref{localsol} and by $\tmax$ its maximal existence time. 
Then we shall see global existence of solutions to the approximate problem \eqref{Pe} 
and their useful estimates. 
We first recall basic inequalities which are often used in studies of the chemotaxis-Navier--Stokes system. 

\begin{lem}\label{pote1}
There exists a constant $C_{1} > 0$ such that 
     \[
     \int_{\Omega} n_{\ep}(\cdot, t) \leq C_1 
\quad \mbox{for all} \ t \in (0, T_{{\rm max}, \ep}) \ \mbox{and all} \  \ep >0.  
     \]
Moreover, there exists $C_{2}>0$ satisfying  
\[
  \int_t^{t+\tau} \int_\Omega n_{\ep}^2 \le C_{2} 
\]
holds for all $t\in (0,T_{{\rm max}, \ep}-\tau)$ 
and all $\ep>0$, 
where 
$\tau \in (0, T_{{\rm max}, \ep})$. 
\end{lem}
\begin{proof}
Integrating the first equation in \eqref{Pe} shows this lemma. 
\end{proof}
 \begin{lem}\label{lem;Linf;c}
The function $t \mapsto \lp{\infty}{c_{\ep}(\cdot, t)}$ is nonincreasing. 
In particular, 
    \begin{equation*}
    \|c_{\ep}(\cdot, t)\|_{L^{\infty}(\Omega)} \le \|c_0\|_{L^{\infty}(\Omega)} 
    \end{equation*}
holds for all $t \in (0, T_{{\rm max}, \ep})$ 
and all $\ep>0$.  
Moreover, we have 
\[
  \int_0^{T_{{\rm max}, \ep}} \int_\Omega 
 |\nabla c_{\ep}|^2 \le \frac 12 \int_\Omega |c_0|^2 
 \quad \mbox{for all}\ \ep>0. 
\]
 \end{lem}
\begin{proof}
Applying the maximum principle 
to the second equation in \eqref{Pe} (see e.g., \cite[Lemma 2.1]{W2014}),  
we can establish 
the $L^\infty$-estimate for $\cep$.  
Moreover, multiplying the third equation in \eqref{Pe} by $\cep$ 
and integrating it over $\Omega \times (0,T_{{\rm max}, \ep})$ imply 
this lemma.  
\end{proof}

We next consider an estimate for 
the energy function $\mathcal{F}_\ep:(0,\tmax)\to \mathbb{R}$ 
defined as 
\[
  \mathcal{F}_\e(t) := \io \nep(\cdot,t) \log \nep(\cdot,t) 
  + \frac \chi 2 \io \frac{|\nabla \cep(\cdot,t)|^2}{\cep(\cdot,t)} 
  + K\chi \io |\uep(\cdot,t)|^2 
\]
with some $K>0$, 
which plays an important role not only in the case that 
$m=1$ (\cite{Lankeit_2016}) 
but also in the case that $m>0$. 
In order to see an estimate for $\mathcal{F}_\e$ we provide the following 3 lemmas. 
\begin{lem}\label{ne}
There exists a constant $C>0$ such that for all $\ep>0$, 
\begin{align*}
\frac{d}{dt}\int_{\Omega} n_{\ep}\log n_{\ep} 
+ \frac{\mu}{2}\int_{\Omega} n_{\ep}^2\log n_{\ep} 
+ \frac{4m}{(m+1)^2}\int_{\Omega} 
\frac{|\nabla (n_{\ep} + \e)^{\frac{m+1}{2}}|^2}{n_{\ep}}
\leq \chi\int_{\Omega}\frac{\nabla n_{\ep} \cdot \nabla c_{\ep}}{1+\ep n_{\ep}} 
     + C 
\end{align*}
on $(0, T_{{\rm max}, \ep})$.
\end{lem}
\begin{proof}
We first obtain from $\nabla \cdot \uep=0$ in $\Omega \times (0,\tmax)$ and  straightforward calculations that 
\begin{align}\label{ineq;energy;n}
\frac d{dt} \io \nep \log \nep 
& = \io \log \nep \Delta (\nep+\e)^m 
- \chi \io \log \nep \nabla \cdot \left( \frac{\nep}{1+\ep\nep}\nabla \cep \right) 
\\ \notag 
&\quad\,  
+ \kappa \io \nep \log \nep -\mu \io \nep^2 \log \nep + \kappa \io \nep -\mu\io \nep^2. 
\end{align}
Then, noting from the boundedness of the functions 
$s\mapsto (\kappa s-\frac \mu 2 s^2)\log s + \kappa s-\mu s^2$ 
on $(0,\infty)$ that 
\begin{align*}
\kappa \io \nep \log \nep -\frac \mu 2 \io \nep^2 \log \nep + \kappa \io \nep -\mu\io \nep^2 
\le C_1 
\end{align*}
with some $C_1>0$,  
we can see from \eqref{ineq;energy;n} with the relation 
\begin{align*}
\int_{\Omega} \log \nep \Delta (\nep+\e)^m 
&= -m\int_{\Omega} (n_{\ep}+\e)^{m-1}\frac{|\nabla n_{\ep}|^2}{n_{\ep}}\\
&= -\frac{4m}{(m+1)^2}\int_{\Omega} 
\frac{|\nabla (n_{\ep}+\e)^{\frac{m+1}{2}}|^2}{n_{\ep}} 
\end{align*}
that this lemma holds. 
\end{proof}

The following 2 lemmas have already been proved 
in the proofs of \cite[Lemmas 2.8 and 2.9]{Lankeit_2016}. 
Thus we only recall statements of lemmas.   
\begin{lem}\label{ushi}
There exist $K, C, k>0$ such that for all $\ep>0$,  
\begin{align*}
\frac{d}{dt}\int_{\Omega}\frac{|\nabla c_{\ep}|^2}{c_{\ep}} 
&+ k\int_{\Omega}c_{\ep}|D^2\log c_{\ep}|^2 
+ k\int_{\Omega}\frac{|\nabla c_{\ep}|^4}{c_{\ep}^3} \\
&\leq C + K\int_{\Omega}|\nabla u_{\ep}|^2 
       -2\int_{\Omega}\frac{\nabla c_{\ep}\cdot\nabla n_{\ep}}{1+\ep n_{\ep}} 
\quad \mbox{on}\ (0, T_{{\rm max}, \ep}).
\end{align*}
\end{lem}
\begin{lem}\label{tra}
For all $\eta>0$ there exists $C_{\eta} > 0$ such that 
for all $\ep>0$,  
\begin{align*}
\frac{d}{dt}\int_{\Omega}|u_{\ep}|^2 + \int_{\Omega} |\nabla u_{\ep}|^2 
\leq \eta \int_{\Omega}n_{\ep}^2\log n_{\ep} + C_{\eta} 
\quad \mbox{on}\ (0, T_{{\rm max}, \ep}).
\end{align*} 
\end{lem}

Thanks to these lemmas, we can establish 
the estimate for $\frac{d \mathcal{F}_\e}{dt}$ 
which enables us to derive the desired estimate for $\mathcal{F}_\e$. 


\begin{lem}\label{tastu}
There exist $C, k_{0}, K >0$ satisfying 
\begin{align}\label{ineq;Fe}
&\frac{d}{dt}\Bigl[\int_{\Omega} n_{\ep}\log n_{\ep} 
+ \frac{\chi}{2}\int_{\Omega}\frac{|\nabla c_{\ep}|^2}{c_{\ep}} 
+ K\chi\int_{\Omega}|u_{\ep}|^2
\Bigr] \\ \notag
&+ \frac{\mu}{4}\int_{\Omega} n_{\ep}^2\log n_{\ep} 
+ \frac{4m}{(m+1)^2}\int_{\Omega}
\frac{|\nabla (n_{\ep}+\e)^{\frac{m+1}{2}}|^2}{n_{\ep}} 
+ k_{0}\int_{\Omega} c_{\ep}|D^2\log c_{\ep}|^2 \\ \notag 
&+ k_{0}\int_{\Omega}\frac{|\nabla c_{\ep}|^4}{c_{\ep}^3} 
+ k_{0}\int_{\Omega}|\nabla u_{\ep}|^2
\leq C \quad \mbox{on}\ (0, T_{{\rm max}, \ep})\quad \mbox{for all}\ \ep>0.
\end{align}
\end{lem}
\begin{proof}
This lemma can be 
derived by a combination of 
Lemmas \ref{ne}, \ref{ushi} and \ref{tra} 
with $\eta:=\frac{\mu}{4K\chi}$. 
\end{proof}

Now we are in a position to see the estimate for $\mathcal{F}_\e$ uniformly-in-$\e$. 
\begin{lem}\label{me}
There exists $C>0$ such that 
\begin{align*}
  \mathcal{F}_\e(t) = \io \nep(\cdot,t) \log \nep(\cdot,t) 
  + \frac \chi 2 \io \frac{|\nabla \cep(\cdot,t)|^2}{\cep(\cdot,t)} 
  + K\chi \io |\uep(\cdot,t)|^2 
\leq C 
\end{align*} 
for all  $t\in (0,\tmax)$ and all $\ep>0$ 
and
\begin{align*}
&\int_{t}^{t+\tau}\int_{\Omega} n_{\ep}^2\log n_{\ep} 
+ \int_{t}^{t+\tau}\int_{\Omega}
\frac{|\nabla (n_{\ep}+\e)^{\frac{m+1}{2}}|^2}{n_{\ep}}
+ \int_{t}^{t+\tau}\int_{\Omega} c_{\ep}|D^2\log c_{\ep}|^2 
\leq C, 
\\[1mm]
&\int_{t}^{t+\tau}\int_{\Omega}\frac{|\nabla c_{\ep}|^4}{c_{\ep}^3} 
+ \int_{t}^{t+\tau}\int_{\Omega}|\nabla u_{\ep}|^2
\leq C, 
\\[1mm]
&\int_{t}^{t+\tau}\int_{\Omega}|\nabla (n_{\ep}+\e)^{\frac{m+1}{2}}|^{\frac{4}{3}} 
+ \int_{\Omega} |\nabla c_{\ep}|^2  
+ \int_{t}^{t+\tau}\int_{\Omega}|\nabla c_{\ep}|^4 
+ \int_{t}^{t+\tau}\int_{\Omega} n_{\ep}^2
\leq C 
\end{align*}
for all $t \in [0, T_{{\rm max}, \ep}-\tau)$ and all $\ep>0$, 
where $\tau:=\min\{1, \frac{1}{2}T_{{\rm max}, \ep}\}$. 
\end{lem}
\begin{proof}
The proof is based on that of \cite[Lemma 2.11]{Lankeit_2016}. 
Noticing from the inequalities $s\log s \le \frac{1}{2e}+s^2 \log s$, 
$\io \frac{|\nabla \cep|^2}{\cep} \le \lp{\infty}{c_0}\io \frac{|\nabla \cep|^4}{\cep^3} + |\Omega|$ 
(from Lemma \ref{lem;Linf;c}) and 
$\io |\uep|^2  \le C_1 \io |\nabla \uep|^2$ 
with some $C_1>0$ (from the Poincar{\'{e}} inequality) 
that Lemma \ref{tastu} implies 
\begin{align*}
\frac {d\mathcal{F}_\e}{dt} + k_1 \mathcal{F}_\e \le k_2
\end{align*}
with some $k_1,k_2>0$, 
we establish the boundedness of $\mathcal{F}_\e$ on $(0,\tmax)$.  
Then for $\tau=\min\{1,\frac 12\tmax\}$ 
integrating the inequality \eqref{ineq;Fe} over $(t,t+\tau)$ with Lemmas \ref{pote1} and 
\ref{lem;Linf;c}
implies this lemma.  
\end{proof}

Then we shall establish global existence in approximate problem \eqref{Pe}  
by using a Moser--Alikakos-type procedure. 

\begin{lem}\label{saru}
For all $\ep\in(0, 1)$, $T_{{\rm max}, \ep}=\infty$ holds.
\end{lem}
\begin{proof} 
The proof is based on that of \cite[Lemma 3.9]{Winkler_2016}. 
Assume that $T_{{\rm max}, \ep}<\infty$ and put $p:=\min\{3+m,4\}$. 
We shall first verify the $L^p$-estimate for $\nep$. 
We see from the first equation and the fact 
$\nabla \cdot \uep =0 $ on $\Omega \times (0,\infty)$ that 
\begin{align*}
\frac{1}{p}\frac{d}{dt}\int_{\Omega}n_{\ep}^{p} 
&= \int_{\Omega}n_{\ep}^{p-1} 
     \nabla\cdot \left( m(n_{\ep}+\e)^{m-1}\nabla n_{\ep}
     -\chi \frac{n_{\ep}}{1+\ep n_{\ep}}\nabla c_{\ep} \right) \\
& \quad \,  
+ \io \nep^{p-1} (\kappa n_{\ep}-\mu n_{\ep}^2)  
- \frac 1p \io u_{\ep}\cdot\nabla n_{\ep}^p  \\
&=-m(p-1)\int_{\Omega}n_{\ep}^{p-2}(n_{\ep}+\e)^{m-1}|\nabla n_{\ep}|^2 \\
     &\,\quad+\chi(p-1)\int_{\Omega}
                  \frac{n_{\ep}^{p-1}}{1+\ep n_{\ep}}\nabla n_{\ep}\cdot\nabla c_{\ep} 
    +\kappa\int_{\Omega}n_{\ep}^{p} 
   -\mu\int_{\Omega}n_{\ep}^{p+1}. 
\end{align*}
Here, since 
$2p-4+2(1-m)_+<p+1$, 
the Young inequality yields that 
\begin{align*}
\chi(p-1)  &\int_{\Omega} \frac{n_{\ep}^{p-1}}{1+\ep n_{\ep}}
                                                        \nabla n_{\ep}\cdot\nabla c_{\ep}  
\\
&
\le  \frac{\chi(p-1)}{\e}  \int_{\Omega} \nep^\frac{p-2}{2}(\nep+\e)^\frac{m-1}{2}|\nabla n_{\ep}|  
     n_{\ep}^{\frac{p}{2}-1}(\nep+\e)^{-\frac{m-1}{2}}
    |\nabla c_{\ep}| 
\\
&\leq 
\frac{m(p-1)}{2}\int_{\Omega} n_{\ep}^{p-2}(n_{\ep}+\e)^{m-1}|\nabla n_{\ep}|^2  
  +\int_{\Omega} n_{\ep}^{2p-4}(\nep+\e)^{2(1-m)}
  +C_{1}\int_{\Omega}|\nabla c_{\ep}|^4 
\\
&\leq \frac{m(p-1)}{2}\int_{\Omega} n_{\ep}^{p-2}(n_{\ep}+\e)^{m-1}|\nabla n_{\ep}|^2  
  +\frac{\mu}{2}\int_{\Omega}n_{\ep}^{p+1} 
+ C_2
  +C_{1}\int_{\Omega}|\nabla c_{\ep}|^4  
\end{align*}
on $(0, T_{{\rm max}, \ep})$ 
with some $C_{1}=C_1(\ep) > 0$ and $C_2=C_2(\ep)>0$,  
where we used the inequalities 
$(a+b)^r \le 2^{r}(a^r + b^r)$ 
($a,b \ge 0$, $r>0$) 
and $(a+b)^r \le b^r$ 
($a,b \ge 0$, $r\le 0$) 
to obtain the last inequality.  
Therefore we obtain from 
the positivity of $n_{\ep}$ 
that
\begin{align*}
\frac{1}{p}\frac{d}{dt}\int_{\Omega}n_{\ep}^{p} 
\leq C_1\int_{\Omega}|\nabla c_{\ep}|^4
+ \kappa\int_{\Omega}n_{\ep}^{p} 
+ C_2. 
\end{align*}
Thus it follows from Lemma \ref{me} 
that 
\begin{align*}
\int_{\Omega}n_{\ep}^p \leq C_{3} 
\quad \mbox{on}\ (0, T_{{\rm max}, \ep}), 
\end{align*}
where $C_{3}=C_{3}(\e)>0$. 
Then, aided by the $L^2$-estimate for 
$\nabla \uep$ (from a testing argument), 
we can obtain that   
\begin{align*}
\|A^\theta u_{\ep}(\cdot, t)\|_{L^{2}(\Omega)} \leq C_{4}
\end{align*}
for all $t\in (0,\tmax)$ with some $C_4=C_4(\ep)>0$. 
Then the continuous embedding $D(A^\theta) \hookrightarrow L^\infty(\Omega)$ 
implies the $L^\infty$-estimate for $\uep$. 
By using these estimates a standard $L^p$-$L^q$ estimate for 
the Neumann heat semigroup on  bounded domains 
and the inequality 
\[
  \lp{3}{u\cd \nabla c \cd} 
  \le \lp{\infty}{u\cd} \lp{6}{\nabla c\cd}^{\frac{1}{2}} 
  \lp{2}{\nabla c\cd}^{\frac{1}{2}}
\]
and the $L^2$-estimate for $\nabla \cep$ from Lemma \ref{me}
imply the $L^6$-estimate for $\nabla \cep$ 
(cf.\ an argument in the proof of \cite[Lemma 3.10]{CKM}). 
Finally we shall verify the $L^\infty$-estimate for $\nep$. 
Put $\nha (x,t):=\max\{\nep (x,t),s_0\} $ for $(x,t)\in \Omega\times (0,\tmax)$ with some 
$s_0>0$. 
Then we can see from $\nabla \cdot \uep=0$ in $\Omega\times (0,\tmax)$ that 
\begin{align*}
 \frac{d}{dt} \io \nha^p 
 &+ p (p-1) \io (\nep+\ep)^{m-1}\nha^{p-2}|\nabla \nha|^2 
\\
 &\le  p(p-1)\chi \io \nha^{p-2} \frac{n_{\ep}}{1+\ep n_{\ep}}\nabla c_{\ep} \cdot \nabla \nha 
 + \kappa \io \nha^{p-1} \nep
\end{align*}  
on $(0,\tmax)$. 
Thus, noting that 
\[  
\frac{n_{\ep}}{1+\ep n_{\ep}}\nabla c_{\ep} 
\in L^{\infty}(0, T_{{\rm max}, \ep}; L^6(\Omega)), 
\quad 
n_{\ep} 
\in L^{\infty}(0, T_{{\rm max}, \ep}; L^{3+2m}(\Omega)), 
\] 
from a Moser--Alikakos-type procedure 
(see the proof of \cite[Lemma A.1]{Tao-Winkler_2012}),  
we can attain that 
\[
\lp{\infty}{\nep\cd}\le C_5 
\]
for all $t\in (0,\tmax)$ with some $C_5=C_5(\ep) >0$, 
which with extensibility criterion shows $\tmax=\infty$ for each $\e \in (0,1)$. 
\end{proof}

\section{Uniform-in-$\e$ estimates}\label{Sec3}

In this section we collect lemmas which are needed 
to show convergence of solutions of \eqref{Pe} as $\ep \searrow 0$. 
Here the case that $m=1$ has already been dealt with in \cite{Lankeit_2016}. 
Thus we shall consider the case that $m>0$ with $m\neq 1$. 
From Lemma \ref{me} we only know 
not some estimate for $\nabla \nep$ but   
$L^{\frac{4}{3}}_{\rm loc}([0,\infty); W^{1,\frac{4}{3}}(\Omega))$-boundedness of 
$((\nep+\e)^{\frac{m+1}{2}})_{\e\in (0,1)}$. 
However, it seems to be difficult to derive an enough 
regularity of $\pa_t (\nep+\e)^{\frac{m+1}{2}}$ for each $m>0$. 
Therefore we need to establish an estimate for 
$\nabla (\nep+\ep)^\gamma$ with some $\gamma < \frac{m+1}{2}$. 
The following lemma is a cornerstone in the proof of Theorem \ref{mainthm1}.  

\begin{lem}\label{tool1}
For all $T>0$ there exists a constant $C=C(T)>0$ such that  
\begin{align*}
\int_0^T\int_{\Omega}|\nabla (n_{\ep}+\ep)^{\frac{m}{2}}|^2 
\leq C \quad \mbox{on}\ (0, T). 
\end{align*} 
\end{lem}
\begin{proof}
In light of Lemma \ref{me} 
we see that  there exists $C_1=C_1(T)>0$ such that 
\begin{align*}
\int_0^T\int_{\Omega}|\nabla (n_{\ep}+\ep)^{\frac{m}{2}}|^2 
&= 
\int_0^T\int_{\Omega}\frac{|\nabla (n_{\ep}+\ep)^{\frac{m+1}{2}}|^2}{n_{\ep}+\ep} 
\\
&\le 
\int_0^T\int_{\Omega}\frac{|\nabla (n_{\ep}+\ep)^{\frac{m+1}{2}}|^2}{n_{\ep}}
\leq C_1
\end{align*} 
with $C_1=C_1(T)>0$.
\end{proof}

We next confirm the following lemma which will play an important role
in deriving some time regularity of $(\nep+\e)^\frac{m}{2}$.  

\begin{lem}\label{tori}
Let $m>0$ be such that $m\neq 1$. 
Then for all $T>0$ there exists a constant $C=C(T)>0$ such that  for all $\e \in (0,1)$, 
\begin{align*}
\int_{\Omega}(n_{\ep} +\e)^{m} \leq C
\end{align*} 
on $(0,T)$ and 
\begin{align*}
\iio (n_{\ep}+\e)^{2m-3}|\nabla n_{\ep}|^2 \le C
\end{align*} 
hold. 
\end{lem}
\begin{proof}
From the first equation in \eqref{P} with $\nabla \cdot \uep =0$ in $\Omega \times (0,\infty)$ 
we have 
\begin{align*}
\frac{1}{m}\frac{d}{dt}\int_{\Omega}(n_{\ep}+\e)^{m} 
& = -m(m-1)\int_{\Omega}(n_{\ep}+\e)^{2m-3}|\nabla n_{\ep}|^2 
\\ 
&\,\quad 
+(m-1) \chi \int_{\Omega}(n_{\ep}+\e)^{m-2}
        \frac{n_{\ep}}{1+\ep n_{\ep}}\nabla n_{\ep}\cdot\nabla c_{\ep} 
\\ & \quad \, 
+ \int_{\Omega}(n_{\ep}+\e)^{m-1} (\kappa\nep-\mu \nep^2).
\end{align*} 
Here we note from the Young inequality that 
\begin{align}\label{ineq;ab;intnepm}
   \left| \chi (m-1) 
        \int_{\Omega}(n_{\ep}+\e)^{m-2}
        \frac{n_{\ep}}{1+\ep n_{\ep}}\nabla n_{\ep}\cdot\nabla c_{\ep} 
      \right|  
&\leq \chi |m-1| \int_{\Omega}(n_{\ep}+\e)^{m-1}
             |\nabla n_{\ep}||\nabla c_{\ep}| 
\\ \notag 
&\leq \frac{m|m-1|}{2}\int_{\Omega}
(n_{\ep}+\e)^{2m-3}|\nabla n_{\ep}|^2 
\\ \notag
&\quad \, 
+ \int_{\Omega} (n_{\ep}+\e)^2 + C_{1}\int_{\Omega} |\nabla c_{\ep}|^4,
\end{align}
where $C_{1}>0$. 
In the case that $m>1$, 
since $m-1>0$, 
we obtain that 
\begin{align*}
\frac{1}{m}\frac{d}{dt}\int_{\Omega}(n_{\ep}+\e)^{m} 
+& \frac{m(m-1)}{2}\int_{\Omega}
(n_{\ep}+\e)^{2m-3}|\nabla n_{\ep}|^2 \\ 
&\leq  \int_{\Omega} (n_{\ep}+\e)^2 + C_{1}\int_{\Omega} |\nabla c_{\ep}|^4 
     + \kappa\int_{\Omega}(n_{\ep}+\e)^{m}, 
\end{align*}
which together with 
Lemma \ref{me} shows that there is $C_2=C_2(T)>0$ such that 
\begin{align*}
\int_{\Omega}(n_{\ep}+\e)^{m} \leq C_{2}
\end{align*}
on $(0,T)$ for all $\ep\in (0,1)$ 
and 
\begin{align*}
\int_0^T\int_{\Omega}(n_{\ep}+\ep)^{2m-3}|\nabla n_{\ep}|^2 
\le C_2
\end{align*}
for all $\ep\in (0,1)$. 
On the other hand, in the case that $0<m<1$, 
we have from \eqref{ineq;ab;intnepm} that 
\begin{align*}
\frac{1}{m}\frac{d}{dt}\int_{\Omega}(n_{\ep}+\e)^{m} 
& \geq  \frac{m(1-m)}{2}\int_{\Omega}
(n_{\ep}+\ep)^{2m-3}|\nabla n_{\ep}|^2 \\ 
&\quad\,  - \int_{\Omega} (n_{\ep}+\e)^2 - C_{1}\int_{\Omega} |\nabla c_{\ep}|^4 
     - \mu \int_{\Omega}(n_{\ep}+\e)^{m+1}.  
\end{align*} 
Thus, integrating it over $(0,T)$, we derive from applications of the Young inequality 
\begin{align*}
\io (\nep+\e)^m \le m \io (\nep+1)  + (1-m)|\Omega| 
\end{align*}
and 
\begin{align*}
\int_0^T \io (\nep+\e)^{m+1} \le \frac{m+1}{2} \int_0^T \io (\nep+1)^2  + \frac{1-m}{2}|\Omega|T
\end{align*}
with Lemma \ref{pote1} that
\begin{align*}
\frac{m(1-m)}{2}\int_0^T \int_{\Omega}
(n_{\ep}+\ep)^{2m-3}|\nabla n_{\ep}|^2 
\le C_3
\end{align*}
with some $C_3=C_3(T)>0$. 
\end{proof}

In order to see some time regularity of $(\nep + \e)^\frac{m}{2}$ 
we will give the following lemma. 

\begin{lem}\label{tool3}
Let $m>0$ be such that $m\neq1$. 
Then for all $T>0$ there exists a constant $C=C(T)>0$ such that  
for all $\e \in (0,1)$, 
\begin{align*}
\io (\nep+\e)^{m-1} \le C
\end{align*}
on $(0,T)$ and 
\begin{align*}
\int_0^T\int_{\Omega}(n_{\ep}+\ep)^{2m-4}|\nabla n_{\ep}|^2 
\leq C 
\end{align*} 
hold. 
\end{lem}
\begin{proof}
We derive from the first equation in \eqref{Pe} 
and integration by parts that 
\begin{align}\label{Blouson}
\frac{d}{dt}\int_{\Omega}(n_{\ep}+\ep)^{m-1} 
&= -m(m-1)(m-2)\int_{\Omega}(n_{\ep}+\ep)^{2m-4}|\nabla n_{\ep}|^2 
\\ \notag 
&\,\quad+\chi(m-1)(m-2)
                 \int_{\Omega}(n_{\ep}+\ep)^{m-3}\frac{n_{\ep}}{1+\ep n_{\ep}}
                                                                \nabla n_{\ep}\cdot\nabla c_{\ep} 
\\ \notag 
&\,\quad+\kappa(m-1)\int_{\Omega}(n_{\ep}+\ep)^{m-2}n_{\ep} 
           -\mu(m-1)\int_{\Omega}(n_{\ep}+\ep)^{m-2}n_{\ep}^2. 
\end{align}
Now we note from the Young inequality and Lemma \ref{me} that 
\begin{align}\label{ineq;m-1}
& \left| \chi(m-1)(m-2)
                 \int_{\Omega}(n_{\ep}+\ep)^{m-3}\frac{n_{\ep}}{1+\ep n_{\ep}}
                                                                \nabla n_{\ep}\cdot\nabla c_{\ep} 
   \right| 
\\ \notag
&\leq \frac{m|(m-1)(m-2)|}{2}\int_{\Omega}(n_{\ep}+\ep)^{2m-4}|\nabla n_{\ep}|^2 
        + C_1\int_{\Omega}|\nabla c_{\ep}|^2 
\\ \notag
&\leq \frac{m|(m-1)(m-2)|}{2}\int_{\Omega}(n_{\ep}+\ep)^{2m-4}|\nabla n_{\ep}|^2 
        + C_2 
\end{align}
with $C_1, C_2>0$. 
We first consider the cases that $m> 2$ and $0<m<1$. 
Since it follows that $(m-1)(m-2)>0$ and 
that if $m>2$ then  
\begin{align*}
  \kappa(m-1)\int_{\Omega}(n_{\ep}+\ep)^{m-2}n_{\ep} 
           -\mu(m-1)\int_{\Omega}(n_{\ep}+\ep)^{m-2}n_{\ep}^2
  \le 
    \kappa(m-1)\int_{\Omega}(n_{\ep}+\ep)^{m-1}
\end{align*}
and if $0<m<1$ then 
\begin{align*}
  \kappa(m-1)\int_{\Omega}(n_{\ep}+\ep)^{m-2}n_{\ep} 
           -\mu(m-1)\int_{\Omega}(n_{\ep}+\ep)^{m-2}n_{\ep}^2
  \le 
     \mu(1-m)\int_{\Omega}(n_{\ep}+\ep)^{m} \le C_3
\end{align*}
with some $C_3=C_3(T)>0$ (from Lemma \ref{tori}), 
we infer from \eqref{Blouson} that 
\begin{align*}
&\frac{d}{dt}\int_{\Omega}(n_{\ep}+\ep)^{m-1} 
+ \frac{m(m-1)(m-2)}{2}\int_{\Omega}(n_{\ep}+\ep)^{2m-4}|\nabla n_{\ep}|^2 
\\ \notag
&\leq \kappa|m-1|\int_{\Omega}(n_{\ep}+\ep)^{m-1} + C_4 
\end{align*}
with some $C_4=C_4(T)>0$, 
and hence there exists a constant $C_5=C_5(T)>0$ such that for all $\ep\in(0,1)$, 
\begin{align*}
\int_{\Omega}(n_{\ep}+\ep)^{m-1} \le C_5 
\end{align*}
on $(0,T)$ and 
\begin{align*} 
\int_{\Omega}(n_{\ep}+\ep)^{2m-4}|\nabla n_{\ep}|^2 
\leq C_5.
\end{align*}
On the other hand, in the case that $1<m<2$,  
we see from \eqref{Blouson} and \eqref{ineq;m-1} that  
\begin{align*}
\frac{d}{dt}\int_{\Omega}(n_{\ep}+\ep)^{m-1} 
&\geq \frac{m(m-1)(2-m)}{2}\int_{\Omega}(n_{\ep}+\ep)^{2m-4}|\nabla n_{\ep}|^2
       -C_6 
\\ \notag
&\quad\, -\mu(m-1)\int_{\Omega}(n_{\ep}+\ep)^{m} 
\end{align*}
with some $C_6>0$. 
Hence, noticing that 
$
\io (\nep+\ep)^{m-1} \le  (m-1) \io (\nep+1) + (2-m)|\Omega| 
$
and $(m-1)(2-m)>0$, 
we derive from 
Lemmas \ref{pote1} and \ref{tori} that 
\begin{align*}
\frac{m(m-1)(2-m)}{2} 
\int_0^T \int_{\Omega}(n_{\ep}+\ep)^{2m-4}|\nabla n_{\ep}|^2 
\leq C_7
\end{align*}
with some $C_7=C_7(T)>0$. 
Finally, in the case that $m=2$, Lemma \ref{tool1} implies this lemma. 
\end{proof}

The following time regularity of $(\nep+\e)^\frac{m}{2}$ 
will be useful in applying 
a Lions--Aubin type lemma later.

\begin{lem}\label{tool4}
Let $m>0$ be such that $m\neq1$. 
Then for all $T>0$ there exists a constant $C=C(T)>0$ satisfying
\[
  \|\pa_t (n_{\ep}+\ep)^{\frac{m}{2}}\|_{L^1(0, T; (W^{2, 4}_{0}(\Omega))^{*})} 
  \leq C \quad \mbox{for all}\ \ep \in (0, 1). 
\]

\end{lem}
\begin{proof}
Let $T>0$ and let $\psi \in L^{\infty}(0, T; W^{2, 4}_0(\Omega))$. 
The first equation in \eqref{Pe} and integration by parts 
yield that 
\begin{align*}
\int_0^T\int_{\Omega} (\partial_t (n_{\ep}+\ep)^{\frac{m}{2}})\psi 
&= -\int_0^T\int_{\Omega}\psi u_{\ep}\cdot\nabla(n_{\ep}+\ep)^{\frac{m}{2}}
\\ 
&\,\quad-\frac{m^2}{m+1}\left(\frac{m}{2}-1\right)
        \int_0^T\int_{\Omega} \psi
               \frac{\nabla (n_{\ep}+\ep)^{\frac{m+1}{2}}}{(n_{\ep}+\ep)^{\frac{1}{2}}}
                  \cdot (n_{\ep}+\ep)^{m-2}\nabla n_{\ep}
\\ 
&\,\quad-\frac{m^2}{2}\int_0^T\int_{\Omega}
                         (n_{\ep}+\ep)^{m-\frac{3}{2}}\nabla n_{\ep}
                                          \cdot (n_{\ep}+\ep)^{\frac{m-1}{2}}\nabla \psi 
\\ 
&\,\quad+\frac{m\chi}{m+1}\left(\frac{m}{2}-1\right)
                 \int_0^T\int_{\Omega}\frac{n_{\ep}\psi}{(1+\ep n_{\ep})(n_{\ep}+\ep)}
                 \frac{\nabla (n_{\ep}+\ep)^{\frac{m+1}{2}}}{(n_{\ep}+\ep)^{\frac{1}{2}}}
                 \cdot \nabla c_{\ep}  
\\ 
&\,\quad+\frac{m\chi}{2}\int_0^T\int_{\Omega}
                                                (n_{\ep}+\ep)^{\frac{m}{2}}
                                                \frac{n_{\ep}}{(1+\ep n_{\ep})(n_{\ep}+\ep)}
                                                \nabla c_{\ep}\cdot\nabla\psi 
\\ 
&\,\quad+\frac{m\kappa}{2}\int_0^T\int_{\Omega}(n_{\ep}+\ep)^{\frac{m}{2}-1}
                                                                                             n_{\ep}\psi
     -\frac{m\mu}{2}\int_0^T\int_{\Omega}(n_{\ep}+\ep)^{\frac{m}{2}-1}
                                                                                           n_{\ep}^2\psi. 
\end{align*} 
Then, noting from 
the Young inequality that 
\begin{align*}
&\left| 
  \int_0^T\int_{\Omega}\psi u_{\ep}\cdot\nabla(n_{\ep}+\ep)^{\frac{m}{2}}
  \right| 
\le 
  \frac{\|\psi\|_{L^\infty(\Omega\times (0,T))}}{2}
  \left(
  \int_0^T\int_{\Omega} |u_{\ep}|^2 
  + \int_0^T\io |\nabla(n_{\ep}+\ep)^{\frac{m}{2}}|^2
  \right), 
\\
&
\left|
   \int_0^T\int_{\Omega} \psi
   \frac{\nabla (n_{\ep}+\ep)^{\frac{m+1}{2}}}{(n_{\ep}+\ep)^{\frac{1}{2}}}
   \cdot (n_{\ep}+\ep)^{m-2}\nabla n_{\ep}
\right|
\\
&\quad \quad \quad \le 
  \frac{\|\psi\|_{L^\infty(\Omega\times (0,T))}}{2}
  \int_0^T\int_{\Omega}  
  \left( \frac{|\nabla (n_{\ep}+\ep)^{\frac{m+1}{2}}|^2}{n_{\ep}}
  + 
  \iio 
   (n_{\ep}+\ep)^{2m-4}|\nabla n_{\ep}|^2
  \right) 
\end{align*}
and
\begin{align*}
&
\left| 
\int_0^T\int_{\Omega}
                         (n_{\ep}+\ep)^{m-\frac{3}{2}}\nabla n_{\ep}
                                          \cdot (n_{\ep}+\ep)^{\frac{m-1}{2}}\nabla \psi 
\right|
\\
&\quad \quad \quad  
\le \frac{\|\nabla \psi\|_{L^\infty(\Omega\times (0,T))}}{2}
\left( 
 \iio  (n_{\ep}+\ep)^{2m-3}|\nabla n_{\ep}|^2
 + 
 \iio (n_{\ep}+\ep)^{m-1}
\right), 
\\ 
&\left| 
                 \int_0^T\int_{\Omega}\psi\frac{n_{\ep}}{(1+\ep n_{\ep})(n_{\ep}+\ep)}
                 \frac{\nabla (n_{\ep}+\ep)^{\frac{m+1}{2}}}{(n_{\ep}+\ep)^{\frac{1}{2}}}
                 \cdot \nabla c_{\ep}   
\right| 
\\
&\quad \quad \quad 
\le 
\frac{\|\psi\|_{L^\infty(\Omega\times (0,T))}}{2}
\left( 
\int_0^T\int_{\Omega}
                 \frac{|\nabla (n_{\ep}+\ep)^{\frac{m+1}{2}}|^2}{n_{\ep}}
+
 \iio |\nabla c_{\ep}|^2
\right) 
\end{align*}
as well as 
\begin{align*}
&\left| 
\int_0^T\int_{\Omega}
  (n_{\ep}+\ep)^{\frac{m}{2}}
  \frac{n_{\ep}}{(1+\ep n_{\ep})(n_{\ep}+\ep)}
 \nabla c_{\ep}\cdot\nabla\psi 
 \right| 
 \\ & \quad \quad \quad 
 \le 
 \frac{\|\nabla \psi\|_{L^\infty(\Omega\times (0,T))}}{2} 
 \left( 
 \int_0^T\int_{\Omega}
  (n_{\ep}+\ep)^{m} 
  + 
 \iio |\nabla c_{\ep}|^2
 \right) 
 \end{align*} 
 with 
 \begin{align*} 
&\left|
  \frac{m\kappa}{2}\int_0^T\int_{\Omega}(n_{\ep}+\ep)^{\frac{m}{2}-1} n_{\ep}\psi
  - \frac{m\mu}{2}\int_0^T\int_{\Omega}(n_{\ep}+\ep)^{\frac{m}{2}-1}
      n_{\ep}^2\psi 
 \right|
\\ & \quad \quad \quad  \quad 
 \le 
 \|\psi\|_{L^\infty(\Omega\times (0,T))} \left( \iio (\nep+\e)^{\max\{m,2\}} +C_1T\right) 
\end{align*}
for some $C_1>0$ (from the fact that $\frac{m}{2}+1 \le \max\{m,2\}$), 
we obtain from Lemmas \ref{me}, \ref{tool1}, \ref{tori} and \ref{tool3} 
together with 
the continuous embedding $W^{2,4}_0(\Omega) \hookrightarrow W^{1,\infty}(\Omega)$ that 
\begin{align*}
\int_0^T\int_{\Omega} (\partial_t (n_{\ep}+\ep)^{\frac{m}{2}})\psi 
\leq C_2\|\psi\|_{L^{\infty}(0, T; W^{2, 4}(\Omega))} 
\end{align*} 
with some $C_2=C_2(T)>0$. 
Therefore a standard duality argument enables us to see this lemma. 
\end{proof}

We also give the following lemma concerned with time regularities of $\cep$ and $\uep$. 

\begin{lem}\label{cu}
For all $T>0$ there exists $C=C(T)>0$ satisfying
$$
\|(c_{\ep})_t\|_{L^2(0, T; (W^{1, 2}_{0}(\Omega))^{*})} \le C \quad \mbox{and} \quad   
\|(u_{\ep})_ t\|_{L^2(0, T; (W^{1, 3}_{0}(\Omega))^{*})}
\leq C\quad \mbox{for all}\ \ep \in (0, 1).
$$
\end{lem}
\begin{proof}
This lemma can be proved from the same arguments as those 
in the proofs of \cite[Lemmas 2.14 and 2.15]{Lankeit_2016}.
\end{proof}

Finally we give an estimate for $\nabla (\nep+\e)^m$ to see   
convergence of $\iio \nabla (\nep+\e)^m \cdot \nabla \varphi$ 
for all $\varphi\in C^\infty_0(\overline{\Omega}\times [0,\infty))$. 

\begin{lem}\label{inu1}
Let $m>0$ be such that $m\neq 1$. 
Then for all $T>0$ and all $r\in(1, \frac 43]$ 
there exists a constant $C=C(T)>0$ such that  
\begin{align*}
\|\nabla (n_{\ep}+\ep)^m\|_{L^r(0, T; L^r(\Omega))} \leq C 
\quad \mbox{for all}\ \ep \in (0, 1). 
\end{align*} 
\end{lem}
\begin{proof}
Let $r=\frac 43$. 
Since $\frac{r}{2-r}= 2$, the Young inequality yields 
\begin{align*}
\int_{0}^{T}\int_{\Omega}|\nabla (n_{\ep}+\ep)^m|^r 
&= \int_{0}^{T}\int_{\Omega}(n_{\ep}+\e)^{(m-1)r}|\nabla n_{\ep}|^r \\
&\leq \int_{0}^{T}\int_{\Omega}(n_{\ep}+\e)^{\frac{r}{2-r}} 
         + C_1\int_{0}^{T}\int_{\Omega}(n_{\ep}+\e)^{2m-3}|\nabla n_{\ep}|^2 \\
&\le  \int_{0}^{T}\int_{\Omega}(n_{\ep}+1)^{2} 
         + C_1\int_{0}^{T}\int_{\Omega}(n_{\ep}+\e)^{2m-3}|\nabla n_{\ep}|^2
\end{align*}
with some $C_1>0$. 
Therefore 
Lemmas \ref{me} and \ref{tori} lead to this lemma.  
\end{proof}


\section{Convergence: Proof of Theorem \ref{mainthm1}}\label{Sec4}

In this section we consider convergence of solutions of approximate problem \eqref{Pe} 
and then prove Theorem \ref{mainthm1}. 
We first state the following result 
which can be obtained from the previous estimates in Section \ref{Sec3}. 

\begin{lem}\label{convergences}
There exist a sequence 
$(\e_j)_{j\in\mathbb{N}}$ such that $\e_j\searrow 0$ 
as $j\to\infty$ and functions $n, c, u$ such that
  \begin{align*}
        n&\in L^2_{\rm loc}([0,\infty);L^2(\om)), 
         \\
        c&\in L^2_{\rm loc}([0,\infty);W^{1,2}(\om)),
         \\
        u&\in L^2_{\rm loc}([0,\infty);W^{1,2}_{0,\sigma}(\om))
  \end{align*}
and that for all $p\in [1,6)$, 
  \begin{align} 
  \label{conv;nm2}
    (\nep+\e)^\frac{m}{2} &\to n^\frac{m}{2} 
    && \mbox{in }L^2_{\rm loc}([0,\infty);L^p(\Omega))  
    \ \mbox{and a.e.} \ \mbox{in} \  \Omega\times (0,\infty) ,   
    \\
    \label{conv;n}
    n_{\ep}&\to n
    &&\mbox{in }L^2_{\rm loc}([0,\infty);L^2(\om)),
  \\
    \label{conv;c}
    c_\e&\to c
    &&\mbox{in }C^0_{\rm loc}([0,\infty);L^p(\om)), 
  \\
    \label{conv;u}
    u_\e&\to u
    &&\mbox{in }L^2_{\rm loc}([0,\infty);L^p(\om)), 
  \\ \label{conv;nac}
    \na c_\e&\to\na c
    &&\mbox{weakly in} \ L^4_{\rm loc}([0,\infty);L^4(\om)),
  \\ \label{conv;nau}
    \na u_\e&\to\na u
    &&\mbox{weakly in} \ L^2_{\rm loc}([0,\infty);L^2(\om)),
  \\ \label{conv;Yu}
    Y_\e u_\e&\to u
    &&\mbox{in }L^2_{\rm loc}([0,\infty);L^2(\om))
  \end{align}
as $\e=\e_j\searrow0$. 
\end{lem}
\begin{proof}
Let $T>0$. 
Thanks to Lemmas \ref{tool1}, \ref{tori} and \ref{tool4}, 
we have that 
\[
  \left( (\nep+\e)^\frac{m}{2} \right)_{\e\in (0,1)} \ \mbox{is bounded in} \ 
  L^2(0,T; W^{1,2}(\Omega)) 
\]
and 
\[
  \left( \pa_t (\nep+\e)^\frac{m}{2} \right)_{\e\in (0,1)} \ \mbox{is bounded in} \ 
  L^1(0,T;(W^{2,4}_0(\Omega)^\ast). 
\]
Therefore, aided by the compact embedding $W^{1,2}(\Omega) \hookrightarrow L^p(\Omega)$ 
for all $p\in [1,6)$ and the continuous embedding 
$L^p(\Omega) \hookrightarrow (W^{2,4}_0(\Omega))^\ast$,  
we can see from a Lions--Aubin type lemma (see \cite[Corollary 4]{Simon_1987}) that 
$((\nep+\e)^\frac{m}{2})_{\e\in (0,1)}$ is relatively compact in 
$L^2 (0,T; L^p(\Omega))$,  
which means that there are a sequence  
$(\e_j)_{j\in \mathbb{N}} \searrow 0$ and a function 
$v \in L^2(0,T;L^p(\Omega))$ such that 
$(\nep+ \e)^\frac{m}{2} \to v$ in $L^2(0,T;L^p(\Omega))$  
as $\e=\e_j\searrow0$.  
Then by putting $n := v^{\frac{2}{m}}$ we have \eqref{conv;nm2}, 
which yields that $\nep \to n$ a.e.\ in $\Omega\times (0,\infty)$ 
as $\e=\e_j\searrow0$.   
The rest of the proof is mainly based on arguments in the proof of 
\cite[Proposition 2.1]{Lankeit_2016}; thus we will give a short proof. 
Since a uniform bound on $\iio \Phi(\nep^2)$, 
where $\Phi(s):=\frac{s}{2}\log (s)$ for $s>0$, 
derives from the Dunford--Pettis theorem (cf.\ \cite[Lemma IV 8.9]{Dunford-Schwartz})
that $(\nep^2)_{\e\in (0,1)}$ is weakly relatively precompact in 
$L^1(\Omega\times (0,T))$, 
we obtain that there is a subsequence of $(\e_j)_{j\in \mathbb{N}}$ such that 
$\iio n_{\ep}^2 \to \iio n^2$ 
as $\e=\e_j\searrow0$.   
This together with the convergence $\nep \to n$ weakly 
in $L^2(0,T;L^2(\Omega))$ 
as $\e=\e_j\searrow0$   
(from Lemma \ref{me}) yields that \eqref{conv;n} along a further subsequence. 
On the other hand, 
by virtue of Lemmas \ref{lem;Linf;c}, \ref{me} and \ref{cu},  
we can establish that 
$(\cep)_{\e\in (0,1)}$ and $((\cep)_t)_{\e\in (0,1)}$ 
are bounded in $L^\infty (0,T; W^{1,2}(\Omega))$ and 
in $L^2(0,T;(W^{1,2}_0(\Omega))^\ast)$, 
respectively, 
as well as 
$(\uep)_{\e\in (0,1)}$ and 
$((\uep)_t)_{\e\in (0,1)}$  
are bounded in $L^2 (0,T; W^{1,2}_\sigma)$ 
and in $L^2(0,T; (W^{1,3}_\sigma(\Omega))^\ast)$, 
respectively. Thus  
\cite[Corollary 4]{Simon_1987} again implies 
\eqref{conv;c} and \eqref{conv;u}, 
and then Lemma \ref{me} leads to the convergences  
\eqref{conv;nac} and \eqref{conv;nau} 
along a further subsequence.  
Finally, noticing that 
$\lp{2}{Y_\e \uep\cd - u\cd} \to 0$ as $\ep=\ep_j\searrow0$ a.e.\ $t>0$  
and $\lp{2}{Y_\e \uep\cd -u\cd}^2 \le C$ for all $t>0$ and all $\ep>0$ 
in view of Lemma \ref{me}, 
we can establish from the dominated convergence theorem that 
\eqref{conv;Yu} along a further subsequence. 
\end{proof}

We then provide convergence of $\nabla (\nep + \e)^m$ from Lemma \ref{inu1}. 

\begin{lem}\label{nablaconv}
Let $m>0$ be such that $m \neq 1$. 
Then the function $n$ obtained in Lemma \ref{convergences} 
satisfies that 
$n^m \in L^{\frac{4}{3}}_{\rm loc}([0, \infty); W^{1, \frac{4}{3}}(\Omega))$ and 
\[
  \nabla(n_{\ep}+\e)^m \to \nabla n^m \quad \mbox{weakly in}\ 
  L^\frac{4}{3}_{\rm loc}([0, \infty); L^\frac{4}{3}(\Omega))
\] 
as $\e=\e_j\searrow0$. 
\end{lem}
\begin{proof}
Let $T>0$. 
Since the Poincar\'e--Wirtinger inequality 
yields that 
\[
\int_0^T \lp{\frac{4}{3}}{(\nep + \ep)^m}^\frac{4}{3} 
   \le C_1 \int_0^T \lp{\frac{4}{3}}{\nabla (\nep+\ep)^m}^\frac{4}{3} 
      + C_1 \int_0^T |\Omega|^{-\frac{1}{4}} \lp{1}{(\nep + \ep)^m} 
\]
with some $C_1>0$, 
we first note from the Fatou lemma and Lemmas \ref{tori}, \ref{inu1} that  
\begin{align*}
  \int_0^T \lp{\frac{4}{3}}{n^m}^\frac{4}{3} 
  & \le \liminf_{\e \searrow 0} \int_0^T \lp{\frac{4}{3}}{(\nep + \ep)^m}^\frac{4}{3} 
  \le C_2
\end{align*}
with some $C_2=C_2(T)>0$, 
which implies that $n^m \in L^\frac{4}{3}(0,T; L^\frac{4}{3}(\Omega))$. 
We next have from Lemma \ref{inu1} that 
there exist a subsequence of $(\e_j)_{j\in \mathbb{N}}$ obtained in Lemma \ref{convergences}  
(again denoted by $(\ep_j)_{j\in \mathbb{N}}$)
and a function $w\in L^\frac{4}{3}(0,T; L^\frac{4}{3}(\Omega))$ such that 
\[
  \nabla (\nep + \ep)^m \to w 
\quad \mbox{weakly in} \ L^\frac{4}{3}(0,T; L^\frac{4}{3}(\Omega)) 
\] 
as $\ep = \ep_j \searrow 0$.  
In order to verify that $w=\nabla n^m$, 
it is enough to confirm that $(\nep + \ep)^m \to n^m$ 
in $L^1(0,T;L^1(\Omega))$  
as $\ep = \ep_j \searrow 0$.   
Now, since we have already known that 
$(\nep+\e)^m$ is uniform integrable 
(from Lemma \ref{tori}) 
and $(\nep+\e)^m \to n^m$ a.e. in $\Omega\times (0,\infty)$ 
as $\ep = \ep_j \searrow 0$,  
the Vitali convergence theorem entails that 
\[
  (\nep + \ep)^m \to n^m\quad  \mbox{in} \ L^1(0,T;L^1(\Omega))
\]
as $\ep = \e_j \searrow 0$. 
Thanks to this strong convergence, 
we can verify that $w=\nabla n^m$, which together with 
$w\in L^\frac{4}{3}(0,T; L^\frac{4}{3}(\Omega))$ 
shows that   
$n^m \in L^\frac{4}{3}(0,T; W^{1,\frac{4}{3}}(\Omega))$.  
\end{proof}

We will establish global existence of weak solutions to \eqref{P} from convergences 
obtained in Lemmas \ref{convergences} and \ref{nablaconv}. 

\begin{prth1.1}
Let $\vp\in C^{\infty}_0(\overline{\Omega}\times [0,\infty))$ and 
$\psi\in C^{\infty}_{0,\sigma}(\Omega \times [0,\infty))$. 
Testing each equations in \eqref{Pe} by these functions 
and using integration by parts, we can see that 
  \begin{align}\label{weaksolPep}
    &-\int^\infty_0\!\!\!\!\int_\Omega \nep \vp_t
     -\int_\Omega n_0\vp(\cdot,0)
     -\int^\infty_0\!\!\!\!\int_\Omega \nep \uep \cdot\na\vp
   \\ \notag
    &\h{5.8mm}=-\int^\infty_0\!\!\!\!\int_\Omega\na (\nep+ \e)^{m} \cdot\na\vp
      +\chi\int^\infty_0\!\!\!\!\int_\Omega \frac{\nep}{1+\e \nep} \na \cep \cdot\na\vp
      +\int^\infty_0\!\!\!\!\int_\Omega (\kappa \nep - \mu \nep^2)\vp,
   \\[2.0mm] \label{weaksolPep2}
    &-\int^\infty_0\!\!\!\!\int_\Omega \cep \vp_t
     -\int_\Omega c_0\vp(\cdot,0)
     -\int^\infty_0\!\!\!\!\int_\Omega \cep \uep \cdot\na\vp
   \\ \notag
    &\h{5.8mm}=-\int^\infty_0\!\!\!\!\int_\Omega \na \cep \cdot \na\vp
      -\int^\infty_0\!\!\!\!\int_\Omega \frac{1}{\e}(\log(1+\e \nep)) \cep\vp,
   \\[2.0mm] \label{weaksolPep3}
    &-\int^\infty_0\!\!\!\!\int_\Omega \uep \cdot\psi_t
     -\int_\Omega u_0\cdot\psi(\cdot,0)
     -\int^\infty_0\!\!\!\!\int_\Omega Y_\e \uep \otimes \uep\cdot\na\psi
   \\ \notag
    &\h{5.8mm}=-\int^\infty_0\!\!\!\!\int_\Omega\na \uep \cdot\na\psi
     +\int^\infty_0\!\!\!\!\int_\Omega \nep \na\psi\cdot\na\Phi
  \end{align}
hold. 
Now, since the dominated convergence theorem implies that 
\begin{align*}
  \frac{1}{1+\e \nep} \to 1 
  \quad \mbox{in} \ L^4_{\rm loc} ([0,\infty); L^4(\Omega)) 
  \quad \mbox{as}\ \ep=\ep_j\searrow0, 
\end{align*} 
the convergences of 
$\nep$ in $L^2_{\rm loc}([0,\infty); L^2(\Omega))$ and 
$\nabla \cep$ weakly in $L^4_{\rm loc}([0,\infty); L^4(\Omega))$ 
(see Lemma \ref{convergences}) 
derive  
\begin{equation}\label{conv;nnabalc}
\frac{\nep}{1+\e \nep}\nabla \cep 
= \nep \cdot \frac{1}{1+\e\nep} \cdot \nabla \cep 
\to n\nabla c \quad \mbox{weakly in} \ L^1_{\rm loc}([0,\infty);L^1(\Omega)) 
\end{equation}
as $\ep=\ep_j\searrow0$. 
On the other hand, 
to confirm 
convergence of $\frac{1}{\e} (\log (1+\e \nep)) \cep$ in 
$L^1_{\rm loc}([0,\infty); L^1(\Omega))$  
we shall show that $f_\e (\nep) \to n $ 
in $L^2_{\rm loc}([0,\infty);L^2(\Omega))$ as $\ep=\e_j \searrow 0$, 
where $f_\e (s):= \frac{1}{\e} \log (1+\e s)$ for $s\geq0$. 
Noticing from \eqref{conv;n} that 
\[ 
  |f_\e (n) - n|^2 \to 0 \quad \mbox{a.e.\ in} \ \Omega\times(0, T) 
  \ \mbox{as} \ \ep=\e_j \searrow 0
\]
and from the inequality $f_\e(s) \leq s$ ($s\geq0$) that 
\[
  |f_\e (n) - n|^2 
    \le  2 n^2,  
\]
we deduce from the dominated convergence theorem that 
for all $T>0$, 
\[
  \|f_\e (n) - n\|^2_{L^2(0,T; L^2(\Omega))} 
  = \int_0^T\int_{\Omega} |f_\e (n) - n|^2  \to 0
\]
as $\e=\e_j \searrow 0$.  
Therefore we can see from the inequality 
\[ 0 < f_\e' (s) = \frac{1}{1+\e s} \le 1\] 
and \eqref{conv;n} that for all $T>0$, 
\begin{align*}
  \|f_\e (\nep) - n\|_{L^2(0,T; L^2(\Omega))} 
  & \le 
  \|f_\e (\nep) - f_\e (n)\|_{L^2(0,T; L^2(\Omega))}
  +\|f_\e (n) - n\|_{L^2(0,T; L^2(\Omega))}
\\ 
  &\le f_\e ' (\nep) 
   \|\nep- n\|_{L^2(0,T; L^2(\Omega))}
  +\|f_\e (n) - n\|_{L^2(0,T; L^2(\Omega))} 
\\ 
  & \to 0
\end{align*}
as $\ep= \ep_j \searrow 0$.  
This together with \eqref{conv;c} enables us to obtain that 
\begin{equation}\label{conv;lognc}
  \frac{1}{\e} (\log (1+\e \nep)) \cep = f_\e (\nep) \cep \to nc 
  \quad \mbox{in} \ L^1_{\rm loc}([0,\infty);L^1(\Omega)) 
\end{equation}
as $\ep= \ep_j \searrow 0$. 
Then all convergences in Lemmas \ref{convergences}, \ref{nablaconv} 
as well as \eqref{conv;nnabalc}, \eqref{conv;lognc} 
make sure to pass to the limit in all integrals in \eqref{weaksolPep}--\eqref{weaksolPep3},  
which means that the triplet $(n,c,u)$ is a global weak solution of \eqref{P}. \qed
\end{prth1.1}



\begin{thebibliography}{10}

\bibitem{Xinru_higher}
X.~Cao.
\newblock Global bounded solutions of the higher-dimensional {K}eller--{S}egel
  system under smallness conditions in optimal spaces.
\newblock {\em Discrete Contin.\ Dyn.\ Syst.}, 35:1891--1904, 2015.

\bibitem{CKM}
X.~Cao, S.~Kurima, and M.~Mizukami.
\newblock Global existence and asymptotic behavior of classical solutions for a
  {3D} two-species chemotaxis-{S}tokes system with competitive kinetics.
\newblock {\em Math.\ Methods Appl.\ Sci.}, to appear, arXiv:1703.01794
  [math.AP].

\bibitem{Chung-Kang_2016}
Y.~Chung and K.~Kang.
\newblock Existence of global solutions for a chemotaxis-fluid system with
  nonlinear diffusion.
\newblock {\em J. Math.\ Phys.}, 57:19 pp, 2016.

\bibitem{Dunford-Schwartz}
N.~Dunford and J.~T. Schwartz.
\newblock {\em Linear Operators. I. General Theory}.
\newblock Pure and Applied Mathematics, 1958.

\bibitem{Hashira-Ishida-Yokota}
T.~Hashira, S.~Ishida, and T.~Yokota.
\newblock Finite-time blow-up for quasilinear degenerate {K}eller--{S}egel
  systems of parabolic--parabolic type.
\newblock {\em J. Differential Equations}, to appear.

\bibitem{Hillen_Painter_2009}
T.~Hillen and K.~J. Painter.
\newblock A user's guide to {PDE} models for chemotaxis.
\newblock {\em J. Math.\ Biol.}, 58:183--217, 2009.

\bibitem{HKMY_1}
M.~Hirata, S.~Kurima, M.~Mizukami, and T.~Yokota.
\newblock Boundedness and stabilization in a two-dimensional two-species
  chemotaxis-{N}avier--{S}tokes system with competitive kinetics.
\newblock {\em J. Differential Equations}, 263:470--490, 2017.

\bibitem{Horstmann-Wang}
D.~Horstmann and G.~Wang.
\newblock Blow-up in a chemotaxis model without symmetry assumptions.
\newblock {\em Eur.\ J.\ Appl.\ Math.}, 12:159--177, 2001.

\bibitem{Ishida-Seki-Yokota}
S.~Ishida, K.~Seki, and T.~Yokota.
\newblock Boundedness in quasilinear {K}eller--{S}egel systems of
  parabolic--parabolic type on non-convex bounded domains.
\newblock {\em J. Differential Equations}, 256:2993--3010, 2014.

\bibitem{Ishida-Yokota_2013}
S.~Ishida and T.~Yokota.
\newblock Blow-up in finite or infinite time for quasilinear degenerate
  {K}eller--{S}egel systems of parabolic--parabolic type.
\newblock {\em Discrete Contin.\ Dyn.\ Syst.\ Ser.\ B}, 18:2569--2596, 2013.

\bibitem{Jin-2017}
C.~Jin.
\newblock Boundedness and global solvability to a chemotaxis model with
  nonlinear diffusion.
\newblock {\em J. Differential Equations}, 263:5759--5772, 2017.

\bibitem{K-S}
E.~F. Keller and L.~A. Segel.
\newblock Initiation of slime mold aggregation viewed as an instability.
\newblock {\em J. Theor.\ Biol.}, 26:399--415, 1970.

\bibitem{lankeit_evsmoothness}
J.~Lankeit.
\newblock Eventual smoothness and asymptotics in a three-dimensional chemotaxis
  system with logistic source.
\newblock {\em J. Differential Equations}, 258:1158--1191, 2015.

\bibitem{Lankeit_2016}
J.~Lankeit.
\newblock Long-term behaviour in a chemotaxis-fluid system with logistic
  source.
\newblock {\em Math.\ Models Methods Appl.\ Sci.}, 26:2071--2109, 2016.

\bibitem{Lankeit-Wang_2017}
J.~Lankeit and Y.~Wang.
\newblock Global existence, boundedness and stabilization in a high-dimensional
  chemotaxis system with consumption.
\newblock {\em Discrete Contin.\ Dyn.\ Syst.}, 37:6099--6121, 2017.

\bibitem{Mizoguchi-Winkler}
N.~Mizoguchi and M.~Winkler.
\newblock Blow-up in the two-dimensional parabolic {K}eller--{S}egel system.
\newblock preprint.

\bibitem{Nagai-Senba-Yoshida}
T.~Nagai, T.~Senba, and K.~Yoshida.
\newblock Application of the {T}rudinger-{M}oser inequality to a parabolic
  system of chemotaxis.
\newblock {\em Funkcial.\ Ekvac.}, 40:411--433, 1997.

\bibitem{OTYM}
K.~Osaki, T.~Tsujikawa, A.~Yagi, and M.~Mimura.
\newblock Exponential attractor for a chemotaxis-growth system of equations.
\newblock {\em Nonlinear Anal.}, 51:119--144, 2002.

\bibitem{Osaki-Yagi}
K.~Osaki and A.~Yagi.
\newblock Finite dimensional attractor for one-dimensional {K}eller--{S}egel
  equations.
\newblock {\em Funkcial.\ Ekvac.}, 44:441--469, 2001.

\bibitem{Simon_1987}
J.~Simon.
\newblock Compact sets in the space {$L^p(0,T;B)$}.
\newblock {\em Ann.\ Mat.\ Pura Appl.}, 146:65--96, 1987.

\bibitem{Tao-Winkler_2011_non}
Y.~Tao and M.~Winkler.
\newblock A chemotaxis-haptotaxis model: the roles of nonlinear diffusion and
  logistic source.
\newblock {\em SIAM J. Math.\ Anal.}, 43:685--704, 2011.

\bibitem{Tao-Winkler_2012}
Y.~Tao and M.~Winkler.
\newblock Boundedness in a quasilinear parabolic-parabolic {K}eller--{S}egel
  system with subcritical sensitivity.
\newblock {\em J. Differential Equations}, 252:692--715, 2012.

\bibitem{Tao-Winkler_2012_KSF}
Y.~Tao and M.~Winkler.
\newblock Global existence and boundedness in a {K}eller--{S}egel-{S}tokes
  model with arbitrary porous medium diffusion.
\newblock {\em Discrete Contin.\ Dyn.\ Syst.}, 32:1901--1914, 2012.

\bibitem{win_aggregationvs}
M.~Winkler.
\newblock Aggregation vs.\ global diffusive behavior in the higher-dimensional
  {K}eller--{S}egel model.
\newblock {\em J. Differential Equations}, 248:2889--2905, 2010.

\bibitem{Winkler_2010_logistic}
M.~Winkler.
\newblock Boundedness in the higher-dimensional parabolic-parabolic chemotaxis
  system with logistic source.
\newblock {\em Comm.\ Partial Differential Equations}, 35:1516--1537, 2010.

\bibitem{W-2012}
M.~Winkler.
\newblock Global large-data solutions in a chemotaxis-({N}avier--){S}tokes
  system modeling cellular swimming in fluid drops.
\newblock {\em Comm.\ Partial Differential Equations}, 37:319--351, 2012.

\bibitem{Winkler_2013_blowup}
M.~Winkler.
\newblock Finite-time blow-up in the higher-dimensional parabolic-parabolic
  {K}eller--{S}egel system.
\newblock {\em J. Math.\ Pures Appl.}, 100:748--767, 2013.

\bibitem{W2014}
M.~Winkler.
\newblock Global asymptotic stability of constant equilibria in a fully
  parabolic chemotaxis system with strong logistic dampening.
\newblock {\em J. Differential Equations}, 257:1056--1077, 2014.

\bibitem{Winkler_2016}
M.~Winkler.
\newblock Global weak solutions in a three-dimensional
  chemotaxis-{N}avier--{S}tokes system.
\newblock {\em Ann.\ Inst.\ H.\ Poincar{\'e} Anal.\ Non Lin{\'e}aire},
  33:1329--1352, 2016.

\bibitem{Zheng-Wang_2016}
J.~Zheng and Y.~Wang.
\newblock Boundedness and decay behavior in a higher-dimensional quasilinear
  chemotaxis system with nonlinear logistic source.
\newblock {\em Comput.\ Math.\ Appl.}, 72:2604--2619, 2016.

\end{thebibliography}

\end{document}